\documentclass[11pt]{amsart}
\usepackage{psfrag}
\usepackage{subfigure}
\usepackage{amsthm}
\usepackage{graphicx}
\usepackage{hyperref} 

\begin{document}
%
%
\newtheorem{theorem}{Theorem}
\newtheorem{algorithm}[theorem]{Algorithm}
\newtheorem{conjecture}[theorem]{Conjecture}
\newtheorem{axiom}[theorem]{Axiom}
\newtheorem{corollary}[theorem]{Corollary}
\newtheorem{definition}[theorem]{Definition}
\newtheorem{example}[theorem]{Example}
\newtheorem{question}[theorem]{Question}
\newtheorem{lemma}[theorem]{Lemma}
\newtheorem*{case1}{Case 1}
\newtheorem*{case2}{Case 2}
\newtheorem{fact}[theorem]{Fact}
\newtheorem{proposition}[theorem]{Proposition}
\newcounter{case_counter}\setcounter{case_counter}{1}
\newtheorem{claim}{Claim}[case_counter]
\newtheorem*{claim1}{Claim 1}
\newtheorem*{claim2}{Claim 2}

%
%
\newcommand{\OMIT}[1]{}
\def\N{\mathbb{N}}
\def\R{\mathbb{R}}
\def\Prob{\mathbb{P}}
\def\P{\mathcal{P}}
\def\C{\mathcal{C}}
\def\U{\mathcal{U}}
\def\E{\mathbb{E}}
\def\eul{\rm{e}}
\def\e{\rm{e}}
\renewcommand{\Gamma}{\varGamma}
\renewcommand{\epsilon}{\varepsilon}
\newcommand{\eps}{\varepsilon}
\newcommand{\floor}[1]{\lfloor{#1}\rfloor}
\newcommand{\ceil}[1]{\lceil{#1}\rceil}
\renewcommand{\bar}{\overline}
\renewcommand{\hat}{\widehat}
\newcommand{\ie}{i.e.\ }
\newcommand{\etc}{etc.\ }
\newcommand{\wrt}{w.r.t.\ }
\newcommand{\whp}{w.h.p.\ }
\newcommand{\sm}{\setminus}
\newcommand{\nib}[1]{\noindent {\bf #1}}
\newcommand{\cproof}[2]{\nib{#1} #2 \qed\medskip}

%
%
\def\COMMENT#1{}
\def\TASK#1{}
\def\noproof{{\unskip\nobreak\hfill\penalty50\hskip2em\hbox{}\nobreak\hfill%
       $\square$\parfillskip=0pt\finalhyphendemerits=0\par}\goodbreak}
\def\endproof{\noproof\bigskip}
\newdimen\margin   
\def\textno#1&#2\par{%
   \margin=\hsize
   \advance\margin by -4\parindent
          \setbox1=\hbox{\sl#1}%
   \ifdim\wd1 < \margin
      $$\box1\eqno#2$$%
   \else
      \bigbreak
      \hbox to \hsize{\indent$\vcenter{\advance\hsize by -3\parindent
      \sl\noindent#1}\hfil#2$}%
      \bigbreak
   \fi}
\def\proof{\removelastskip\penalty55\medskip\noindent{\bf Proof. }}
\def\enddiscard{}
\long\def\discard#1\enddiscard{}

%
%
\date{\today}
\title{Cycles Of Given Length In Oriented Graphs}
\author{Luke Kelly, Daniela K\"uhn \and Deryk Osthus}
\thanks {D.~K\"uhn and D.~Osthus were supported by the EPSRC, grant no.~EP/F008406/1.}
\begin{abstract}
We show that for each $\ell\geq 4$ every sufficiently large oriented graph~$G$ with
$\delta^+(G),\delta^-(G)\ge \lfloor|G|/3\rfloor+1$
contains an $\ell$-cycle. This is best possible for all those $\ell\geq 4$ which are not
divisible by~3.
Surprisingly, for some other values of~$\ell$, an $\ell$-cycle is forced by a much
weaker minimum degree condition.
We propose and discuss a conjecture regarding the precise minimum degree which forces
an $\ell$-cycle (with $\ell \ge 4$ divisible by $3$) in an oriented graph.
We also give an application of our results to pancyclicity and consider $\ell$-cycles in general digraphs.
\end{abstract}
\maketitle

\section{Introduction}\label{sec:intro}

\subsection{Girth}

All the directed graphs (digraphs) considered in this paper have no loops and at most two edges between each
pair of vertices: at most one edge in each direction. A digraph is an \emph{oriented graph}
if it is an orientation of a simple graph.
A central problem in digraph theory is the Caccetta-H\"aggkvist conjecture~\cite{CH}
(which generalized an earlier conjecture of Behzad, Chartrand and Wall~\cite{BCW}):%
\COMMENT{which stated that every regular strong digraph~$D$
with minimum semidegree~$|D|/\ell$ contains a cycle of length at most~$\ell$.}
\begin{conjecture} \label{CHC}
An oriented graph on $n$ vertices with minimum outdegree~$d$
contains a cycle of length at most $\lceil n/d\rceil$.
\end{conjecture}
Note that in Conjecture~\ref{CHC} it does not matter whether we consider oriented graphs or
general digraphs.
Chv\'atal and Szemer\'edi~\cite{chvatal_szemeredi} showed that a minimum outdegree of at least~$d$ guarantees
a cycle of length at most $\lceil 2n/(d+1)\rceil$. For most values of $n$ and $d$, this is improved
by a result of Shen~\cite{shen_2}, which guarantees a cycle of length at most $3\lceil 0.44 n/d\rceil$.
Chv\'atal and Szemer\'edi~\cite{chvatal_szemeredi} also showed that Conjecture~\ref{CHC} holds
if we increase the bound on the cycle length by adding a constant~$c$. They showed that $c:=2500$ will do.
Nishimura~\cite{Nishimura} refined their argument to show that one can take
$c:=304$. The next result of Shen gives the best known constant.
\begin{theorem}[Shen~\cite{shen_ch}]\label{thm:shen}
An oriented graph on $n$ vertices with minimum outdegree~$d$
contains a cycle of length at most $\lceil n/d\rceil +73$.
\end{theorem}
The special case of Conjecture~\ref{CHC} that has attracted most interest is when $d=\lceil n/3 \rceil$.
The following bound towards this case improves an earlier one
of Caccetta and H\"aggkvist~\cite{CH}.
\begin{theorem}[Shen~\cite{shen_triangle}]\label{shen_triangle}
If $G$ is any oriented graph on $n$ vertices with
$\delta^+(G) \geq 0.355n$ then $G$ contains a directed triangle.
\end{theorem}
If one considers the \emph{minimum semidegree} $\delta^0(G):=\min \{\delta^+(G),\delta^-(G) \}$ instead of
the minimum outdegree $\delta^+(G)$, then the constant can be improved slightly. The best known
value for the constant in this case is currently $0.346$~\cite{HHK_CH}.
See the monograph~\cite{BGbook} or the survey~\cite{Nathanson} for further partial results
on Conjecture~\ref{CHC}.

\subsection{Cycles of given length in oriented graphs}

We consider the natural and related question of which minimum
semidegree forces cycles of length exactly $\ell\geq 4$ in
an oriented graph. We will often refer to cycles of length~$\ell$ as \emph{$\ell$-cycles}.
Our main result answers this question completely when~$\ell$ is not a multiple of~$3$.
\begin{theorem}\label{thm:short_cycles}
Let $\ell\geq 4$. If $G$ is an oriented graph on $n \ge 10^{10}\ell$ vertices with
$\delta^0(G) \ge \lfloor n/3 \rfloor + 1$ then $G$ contains an $\ell$-cycle.
Moreover for any vertex $u\in V(G)$ there is an $\ell$-cycle containing~$u$.
\end{theorem}
The extremal example showing this to be best possible for $\ell\geq 4$, $\ell\not\equiv 0 \mod 3$
is given by the blow-up of a $3$-cycle. More precisely,
let $G$ be the oriented graph on $n$ vertices formed by dividing $V(G)$ into $3$
vertex classes $V_1,V_2,V_3$ of as equal size as possible and
adding all possible edges from $V_i$ to $V_{i+1}$, counting modulo~$3$.
Then this oriented graph contains no $\ell$-cycle and has minimum semidegree~$\floor{n/3}$.

Also, for all those $\ell \ge 4$ which are multiples of $3$, the `moreover' part is best possible
for infinitely many~$n$. To see this,
consider the modification of the above example formed by
deleting a vertex from the largest vertex class and
adding an extra vertex $u$ with $N^+(u)=V_2$ and
$N^-(u)=V_1$. This gives an oriented graph with minimum semidegree
$\lfloor (n-1)/3\rfloor$. For $\ell \equiv 0 \mod 3$ it contains no $\ell$-cycle through~$u$.

Perhaps surprisingly, we can do much better than Theorem~\ref{thm:short_cycles} for some cycle lengths
(if we do not ask for a cycle through a given vertex).
Indeed, we conjecture that the correct bounds are those given by the obvious extremal example:
when we seek an $\ell$-cycle, the extremal example is probably the blow-up of a $k$-cycle, where~$k \ge 3$
is the smallest integer which is not a divisor of~$\ell$.

\begin{conjecture}\label{con:short}
Let $\ell\geq 4$ be a positive integer and let~$k$ be the
smallest integer that is greater than~$2$ and does not
divide~$\ell$.
Then there exists an integer $n_0=n_0(\ell)$ such that every oriented
graph~$G$ on $n\geq n_0$ vertices with minimum semidegree $\delta^0(G)\geq \lfloor n/k\rfloor+1$
contains an $\ell$-cycle.
\end{conjecture}
It is easy to see that
the only values of $k$ that can appear in Conjecture~\ref{con:short} are of the form
$k=p^s$ with $k \ge 3$, where $p \ge 2$ is a prime and $s$ a positive integer.%
     \COMMENT{Suppose that $k=p^s q$, where $p \ge 2$ and $q \ge 2$ is coprime to $p$.
Wlog $p^s\ge q$. Also, we may assume $k >4$ and so $p^s \ge 3$ since $p^s\ge q$.
Suppose that $q \ge 3$.
Then $\ell$ is a multiple
of $p^s$ and $q$ and so of $p^s q$, a contradiction.
So suppose $q=2$. Then $k>4$ implies that $\ell$ is a multiple of $4$ and thus of $2$.
Thus $\ell$ is also a multiple of $p^s q$.}
Theorem~\ref{thm:short_cycles} confirms this conjecture in the case when~$k=3$.
The following result implies that Conjecture~\ref{con:short} is approximately true when $k=4,5$
and~$\ell$ is sufficiently large. It also gives weaker bounds on the minimum semidegree
for large values of~$k$.
\begin{theorem}\label{thm:weakconj}
Let $\ell\geq 4$ be a positive integer and let~$k$ be the
smallest integer that is greater than~$2$ and does not
divide~$\ell$.
\begin{enumerate}
\item[{\rm (i)}] There exists an integer $n_0=n_0(\ell)$ such that whenever $k\ge 150$
and~$G$ is an oriented graph on $n\geq n_0$ vertices with $\delta^+(G)\ge n/k+150n/k^2$
then~$G$ contains an $\ell$-cycle.
\item[{\rm (ii)}] If $k=4$ and $\ell\ge 42$ then for every $\eps>0$ there exists an
integer $n_0=n_0(\ell,\eps)$ such that every oriented graph~$G$ on $n\geq n_0$ vertices
with $\delta^0(G)\ge n/k+\eps n$ contains an $\ell$-cycle.
\item[{\rm (iii)}] The analogue of (ii) holds if $k=5$ and $\ell\ge 2550$.
\end{enumerate}
\end{theorem}%
   \COMMENT{With careful analysis it can be shown that
in (ii) only 6, 9, 12, 15, 18, 27 and 33
are not possible with our method. All that is really needed is that for every
pair $t$, $t+1$ with $2\leq t\leq 6$ we can express $\ell$ as a positive linear
sum of $t$ and $t+1$. }
Part~(i) is obtained from Theorem~\ref{thm:shen} via a simple application of the
Regularity lemma for digraphs (see Section~\ref{partial}).
It would be interesting to find a proof which does not rely on the
Regularity lemma. Moreover, part~(i) suggests that one might be able to replace~$\delta^0$
by~$\delta^+$ in Conjecture~\ref{con:short}. Even replacing it in Theorem~\ref{thm:short_cycles} would
be interesting.

In view of Theorem~\ref{thm:short_cycles} and the Caccetta-H\"aggkvist Conjecture
one might wonder whether a minimum semidegree close to $n/3$ also forces a $3$-cycle
through any given vertex. However the next proposition
(whose straightforward proof is given in Section~\ref{sec:short})
shows that the threshold in this case is much higher.

\begin{proposition}\label{prop:3cycles}
\mbox{}\begin{itemize}
\item[(i)] If $G$ is an oriented graph on $n$ vertices with $\delta^0(G)\geq \lceil 2n/5 \rceil$ then for any
vertex $u\in V(G)$ there exists a $3$-cycle containing $u$.%
\COMMENT{introduced the $\lceil$ to make readers aware of the difference in (i) and (ii)}
\item[(ii)] For infinitely many $n$ there exists an oriented graph $G$ on $n$ vertices with
$\delta^0(G)=\floor{2n/5}$ containing a vertex $u$ which
does not lie on a $3$-cycle.
\end{itemize}
\end{proposition}

\subsection{Pancyclicity} \label{pancyc}

Building on~\cite{kelly_kuhn_osthus_hc_orient}, Keevash, K\"uhn and Osthus~\cite{KKO_exact} recently
gave an exact minimum semidegree bound which forces a Hamilton cycle in an oriented graph. More precisely,
they showed that every sufficiently large oriented graph~$G$ with $\delta^0(G)\ge (3n-4)/8$
contains a Hamilton cycle.
This is best possible and settles a problem of Thomassen.
The arguments in~\cite{KKO_exact} can easily be modified to
show that~$G$ even contains an $\ell$-cycle for every
$\ell\ge n/10^{10}$ through any given vertex (see~\cite{Lukethesis} for
details). Together with Theorems~\ref{shen_triangle} and~\ref{thm:short_cycles} this implies that~$G$
is \emph{pancyclic}, i.e.~it contains cycles of all possible lengths.%

\begin{theorem}\label{thm:pan}
There exists an integer~$n_0$ such that every oriented graph~$G$ on $n \ge n_0$ vertices with minimum
semidegree $\delta^0(G) \ge (3n-4)/8$ contains an $\ell$-cycle for all $3\le \ell \le n$.
Moreover, if~$4\leq \ell \leq n$ and if~$u$ is any vertex of~$G$ then~$G$ contains an
$\ell$-cycle through~$u$.
\end{theorem}
This improves a bound of Darbinyan \cite{Darbinyan}, who proved
that a minimum semidegree of $\floor{n/2}-1\geq 4$ implies pancyclicity.
Another degree condition which implies pancyclicity in oriented graphs which are close
to being tournaments is given by Song~\cite{Song}.
Proposition~\ref{prop:3cycles} shows that we cannot have~$\ell=3$ in the `moreover' part
of Theorem~\ref{thm:pan}.

For (general) digraphs, Thomassen~\cite{tom} as well as
H\"aggkvist and Thomassen~\cite{haggtom} gave degree conditions
which imply that every digraph with  minimum semidegree $>n/2$
is pancyclic. (The complete bipartite digraph whose vertex
class sizes are as equal as possible shows that the latter
bound is best possible.) Alon and Gutin~\cite{ag} observed that
one can use Ghouila-Houri's theorem~\cite{gh} (which states
that a minimum semidegree of at least $n/2$ guarantees a
Hamilton cycle in a digraph) to show that every digraph $G$
with minimum semidegree $>n/2$ is even vertex-pancyclic,
i.e.~for every $\ell=2,\dots,n$ each vertex of~$G$ lies on an
$\ell$-cycle.

\subsection{Arbitrary orientations of cycles}\label{sec:arbintro}

Recently Kelly~\cite{K_arb} proved the following result on
arbitrary orientations of Hamilton cycles in oriented graphs.

\begin{theorem}\label{thm:arb_ham}
For any~$\alpha>0$ there exists~$n_0=n_0(\alpha)$ such that
every oriented graph~$G$ on~$n\geq n_0$ vertices with minimum
semidegree~$\delta^0(G)\geq (3/8+\alpha)n$ contains every
possible orientation of a Hamilton cycle.
\end{theorem}

In this paper we extend this further to a pancyclicity result
for arbitrary orientations: if an oriented graph~$G$ on~$n$
vertices contains every possible orientation of an~$\ell$-cycle
for all $3\le \ell \le n$ we say that~$G$ is \emph{universally
pancyclic}. Our main result on arbitrary orientations says that
asymptotically universal pancyclicity requires the same minimum
semidegree as pancyclicity.

\begin{theorem}\label{thm:arbpan}
For all~$\alpha>0$ there exists an integer~$n_0=n_0(\alpha)$
such that every oriented graph~$G$ on $n \ge n_0$ vertices with
minimum semidegree $\delta^0(G) \ge (3/8+\alpha)n$ is
universally pancyclic.
\end{theorem}

As with standard orientations, if we look only at short cycles
then we can strengthen the minimum semidegree condition in the above
result. The semidegree required will depend on the so-called cycle-type.
Given an arbitrarily oriented $\ell$-cycle $C$, the
\emph{cycle-type}~$t(C)$ of~$C$ is the number
of edges oriented forwards in~$C$ minus the number of edges
oriented backwards in~$C$. By traversing~$C$ in the opposite direction if necessary,
we may assume that $t(C)\ge 0$. An oriented~$\ell$-cycle
has cycle-type~$\ell$. Arbitrarily oriented cycles of
cycle-type~0 are precisely those for which there is a digraph
homomorphism into an oriented path. Moreover, if $t(C)\geq 3$ then $t(C)$ is the
\emph{maximum} length of an oriented cycle into which there is
a digraph homomorphism of~$C$.

\begin{proposition}\label{prop:arb_short}
\mbox{}\begin{itemize}
\item Let $\ell\geq 4$ and
    let~$\alpha>0$. Then there exists~$n_0=n_0(\ell,\alpha)$ such
    that every oriented graph~$G$ on~$n\geq n_0$ vertices with
    minimum semidegree~$\delta^0(G)\geq (1/3+\alpha)n$ contains
    every orientation of an~$\ell$-cycle.
\item Let~$\alpha>0$ and let~$\ell$ be some positive constant.
    Then there exists~$n_0=n_0(\alpha,\ell)$ such that every
    oriented graph~$G$ on~$n\geq n_0$ vertices with minimum
    semidegree~$\delta^0(G)\geq \alpha n$ contains every
    cycle of length at most~$\ell$ and cycle-type 0.
\end{itemize}
\end{proposition}

In Section~\ref{sec:arb} we will derive the universal pancyclicity result (Theorem~\ref{thm:arbpan})
by combining the short-cycle result
(Proposition~\ref{prop:arb_short}) with a probabilistic
argument applied to Theorem~\ref{thm:arb_ham} giving all long cycles.

Conjecture~\ref{con:short} has a natural strengthening to
incorporate arbitrarily oriented cycles.

\begin{conjecture}\label{con:arb_short}
Let~$C$ be an arbitrarily oriented cycle of length~$\ell\geq 4$
and cycle-type~$t(C)\geq 4$. Let~$k$ be the smallest integer which is greater than~$2$ and
does not divide~$t(C)$. Then there exists an integer
$n_0=n_0(\ell,k)$ such that every oriented graph~$G$ on $n\geq
n_0$ vertices with minimum semidegree $\delta^0(G)\geq \lfloor
n/k\rfloor+1$ contains~$C$.
\end{conjecture}

As we shall see in Section~\ref{sec:arb}, Conjecture~\ref{con:short} would
imply an approximate version of Conjecture~\ref{con:arb_short}.

\subsection{Cycles of given length in digraphs}

A straightforward application of the Regularity lemma shows that a solution to Conjecture~\ref{con:short}
would also asymptotically solve the corresponding problem for general digraphs:
Let $\delta_{di}(\ell,n)$ denote the smallest integer $d$ so that every digraph with $n$ vertices
and minimum semidegree at least~$d$ contains an $\ell$-cycle and let $\delta_{orient}(\ell,n)$
denote the smallest integer $d$ so that every oriented graph with $n$ vertices
and minimum semidegree at least~$d$ contains an $\ell$-cycle.
\begin{proposition} \label{digraphmindeg}
For any $\ell \ge 3$,
$$
\lim_{n \to \infty} \frac{\delta_{di}(\ell,n)}{n}=
\begin{cases}
1/2 & \text{ if } \ell \text{ is odd};\\
\lim_{n \to \infty} \frac{\delta_{orient}(\ell,n)}{n} &\text{ otherwise}.
\end{cases}
$$
\end{proposition}
It is easy to see that these limits exist.\footnote{Suppose for example that
$\lim_{n\to\infty} \delta_{orient}(\ell,n)/n$ does not exist.
Then there is an $\eps>0$ such that for every $n'\in \N$ there exist $n_2>
n_1\ge n'$ with
$c_2:=\delta_{orient}(\ell,n_2)/n_2\ge \delta_{orient}(\ell,n_1)/n_1+\eps=:c_1+\eps$.
Let $G_2$ be any oriented graph on~$n_2$ vertices with $\delta^0(G_2)\ge
c_2n_2-1$ (say) which does not contain an $\ell$-cycle.
Pick a random set~$X\subseteq V(G_2)$ of size~$n_1$.
Then $G_2[X]$ has minimum semidegree at least $(c_2-\eps/2)n_1$,
contradicting the fact that
$\delta_{orient}(\ell,n_1)/n_1=c_1$.}
We will prove Proposition~\ref{digraphmindeg} in Section~\ref{partial}.
The corresponding density
problem for digraphs was solved by H\"aggkvist and Thomassen.
Let $ex_{di}(\ell,n)$ denote the largest number $d$ so that
there is digraph with~$n$ vertices and at least $d$ edges which
contains no $\ell$-cycle. H\"aggkvist and
Thomassen~\cite{haggtom} proved that
\begin{equation}\label{density}
ex_{di}(\ell,n)=\binom{n}{2}+\frac{(\ell-2)n}{2}.
\end{equation}
The case $\ell=3$ was proved earlier by Brown and
Harary~\cite{brownharary}. A transitive tournament (i.e.~an
acyclic orientation of a complete graph) shows that it does not
make sense to consider this density problem for oriented
graphs. More general extremal digraph problems are discussed in
the surveys~\cite{brownsimonovits, survey}.

\section{Notation}\label{sec:notation}

Given two vertices~$x$ and~$y$ of a digraph~$G$, we write~$xy$
for the edge directed from~$x$ to~$y$. The \emph{order}~$|G|$
of~$G$ is the number of its vertices. We write~$N^+_G(x)$ for
the outneighbourhood of a vertex~$x$ and $d^+(x):=|N^+_G(x)|$
for its outdegree. Similarly, we write~$N^-_G(x)$ for the
inneighbourhood of~$x$ and $d^-(x):=|N^-_G(x)|$ for its
indegree. Given $X\subseteq V(G)$ we denote $|N^+_G(x)\cap X|$
by $d^+_X(x)$, and define $d^-_X(x)$ similarly. We write
$N_G(x):=N^+_G(x)\cup N^-_G(x)$ for the neighbourhood of~$x$.
We use~$N^+(x)$ etc.~whenever this is unambiguous. Given a
set~$A$ of vertices of~$G$, we write $N^+_G(A)$ for the set of
all outneighbours of vertices in~$A$. So~$N^+_G(A)$ is the
union of $N^+_G(a)$ over all $a\in A$. $N^-_G(A)$ is defined
similarly. The directed subgraph of~$G$ induced by~$A$ is
denoted by~$G[A]$ and we write~$e(A)$ for the number of its
edges. $G-A$ denotes the digraph obtained from~$G$ by
deleting~$A$ and all edges incident to~$A$.

When referring to paths and cycles in digraphs we always mean
that they are directed without mentioning this explicitly.
Given two vertices $x,y$ of a digraph~$G$, an \emph{$x$-$y$
path} is a directed path which joins~$x$ to~$y$. Given two
subsets~$A$ and~$B$ of vertices of~$G$, an~$A$-$B$ edge is an
edge~$ab$ where $a\in A$ and $b\in B$. We write $e(A,B)$ for
the number of all these edges. A \emph{walk} in~$G$ is a
sequence $v_1v_2\dots v_{\ell}$ of (not necessarily distinct)
vertices, where $v_iv_{i+1}$ is an edge for all $1\leq i<\ell$.
The length of a walk is~$\ell-1$. The walk is \emph{closed} if
$v_1=v_\ell$. Given two vertices $x,y$ of~$G$, the
\emph{distance $dist(x,y)$ from~$x$ to~$y$} is the length of
the shortest $x$-$y$ path. The \emph{diameter} of $G$ is
the maximum distance between any ordered pair of vertices.

\section{Proofs of Theorem~\ref{thm:short_cycles} and Proposition~\ref{prop:3cycles}}\label{sec:short}

We begin with two immediate facts about oriented graphs which will prove very useful.

\begin{fact}\label{fact:avg}
If $G$ is an oriented graph
and $X\subseteq V(G)$ is non-empty then $e(X)\leq |X|(|X|-1)/2$.
In particular, there exists $x\in X$ with $|N^+(x)\cap X|\leq |X|/2-1/2$
and thus
$|N^+(X)\sm X|\geq  |N^+(x)\sm X|\geq \delta^0(G)  -  |X|/2+1/2.$ \qed
\end{fact}

\begin{fact}\label{fact:ind}
If $G$ is an oriented graph on $n$ vertices
then the maximum size of an independent set is at most $n-2\delta^0(G)$. \qed
\end{fact}

\cproof{Proof of Proposition~\ref{prop:3cycles}.}
{First we prove (i). By Fact \ref{fact:avg} there exists a vertex $x\in N^+(u)$ with
\[|N^+(x)\sm N^+(u)|\geq \delta^0(G)-|N^+(u)|/2+1/2.\]
Hence
\[ |N^+(u)|+|N^-(u)|+|N^+(x)\sm N^+(u)|\geq 5\delta^0(G)/2+1/2>n\]
and so~$x$ must have an outneighbour in $N^-(u)$.

For (ii), pick~$m\in\mathbb{N}$ and define an oriented
graph~$G$ on $n:=5m-1$ vertices as follows. Let $A$, $B$, $C$
be disjoint vertex sets of sizes $2m-1$, $2m-1$ and $m$
respectively. Add all possible edges from $A$ to $B$, $B$ to
$C$ and $C$ to $A$. Let $G[A]$ and $G[B]$ induce regular
tournaments. So for example every vertex in $A$ will have~$m-1$
outneighbours and~$m-1$ inneighbours in~$A$. (It is easy to see
that such oriented graphs exist.) Add a single vertex~$u$ with
$N^+(u):=B$ and $N^-(u):=A$. Then
$\delta^0(G)=2m-1=\floor{2n/5}$. By construction $u$ is not
contained in a 3-cycle. }

We now prove Theorem~\ref{thm:short_cycles} in a series of lemmas.
Lemmas \ref{lemma:4}, \ref{lemma:5} and \ref{lemma:6} deal with the special
cases $\ell=4,5,6$. Lemmas \ref{lemma:34or5} and \ref{lemma:7} deal with the general case $\ell\geq 7$.

\begin{lemma}\label{lemma:4}
If $G$ is an oriented graph on $n\geq 4$ vertices with
$\delta^0(G)\geq \lfloor n/3\rfloor+1$ then for any vertex
$x\in V(G)$,~$G$ contains a $4$-cycle through~$x$.%
    \COMMENT{Do not need the $n\geq 4$ condition here, but it fits nicely with $\geq 6$
in the 6-cycle case. }
\end{lemma}

\begin{proof}
Assume that there is a vertex $x\in V(G)$ for which no such cycle exists. Let~$X$ be a
set of $\lfloor n/3\rfloor+1$ outneighbours of~$x$ and~$Y$ be a set of $\lfloor n/3\rfloor+1$ inneighbours.
Suppose that both of the following hold.
\begin{enumerate}
\item[(i)] There exists $x'\in X$ with $|N^+(x')\sm (X\cup Y)|\geq (\lfloor n/3\rfloor+1)/2$.
\item[(ii)] There exists $y'\in Y$ with $|N^-(y')\sm (X\cup Y)|\geq (\lfloor n/3\rfloor+1)/2$.
\end{enumerate}
Then
\[(N^+(x')\cap N^-(y'))\sm (X\cup Y)\neq\emptyset\]
and hence the desired 4-cycle exists. So without loss of generality assume that~(i) does not hold.
(The case when~(ii) does not hold is similar.)
Let $X'$ be the set of vertices $x'\in X$ with $d_X^-(x')>0$.
Note that Fact \ref{fact:ind} implies that $X'\neq\emptyset$.
Let $x'\in X'$ be such that $d_{X'}^+(x')$ is minimal.
Since $N^+(x')\cap (X\sm X')=\emptyset$, Fact~\ref{fact:avg} implies that
\[
|N^+(x')\sm X|= |N^+(x')\sm X'| \geq
	\delta^0(G)-|X'|/2 \geq \delta^0(G)-|X|/2 \geq (\lfloor n/3\rfloor+1)/2.
\]
Since we are assuming that~(i) does not hold this means that~$x'$ has an outneighbour
$y\in Y$. By definition of~$X'$ there exists an inneighbour~$x''\in X$ of~$x'$. But
then $xx''x'y$ is the required 4-cycle.
\end{proof}

\begin{lemma}\label{lemma:5}
If $G$ is an oriented graph on $n\geq 5$ vertices with
$\delta^0(G)\geq \lfloor n/3\rfloor+1$ then for any vertex
$x\in V(G)$, $G$ contains a $5$-cycle through~$x$.
\end{lemma}

\begin{proof}
As $N^-(x)$ is not independent by Fact~\ref{fact:ind} we can
pick vertices $a,y\in N^-(x)$ such that $ya,ax,yx\in E(G)$.
Let~$X$ be a set of $\lfloor n/3\rfloor+1$
outneighbours of~$x$ and~$Y$ be a set of $\lfloor n/3\rfloor+1$ inneighbours of~$y$.
Define $Z:=X\cap Y$. Clearly, it suffices to prove the next claim.

\begin{claim1} There exists at least one of the following:
\begin{enumerate}
	\item[(i)] an $x$-$y$ path of length $4$,
	\item[(ii)] an $x$-$y$ path of length $3$ avoiding $a$.
\end{enumerate}
\end{claim1}
\noindent
Note that $x,y,a\not\in X\cup Y$ since $G$ is an oriented graph.
So we may assume that $e(X,Y)=0$, as otherwise (ii) is satisfied.
In particular,~$Z$ is independent and $e(X,Z)=e(Z,Y)=0$. The following claim immediately
implies~(i) (to see this, note that $x,y \notin N^+(x')\cap N^-(y')$).

\begin{claim2} Both of the following hold.
\begin{enumerate}
\item[(a)] There exists $x'\in X$ with $|N^+(x')\sm (X\cup Y)|\geq (\lfloor n/3\rfloor+1+|Z|)/2$.
\item[(b)] There exists $y'\in Y$ with $|N^-(y')\sm (X\cup Y)|\geq (\lfloor n/3\rfloor+1+|Z|)/2$.
\end{enumerate}
\end{claim2}
\noindent
We will only prove~(a) (the argument for~(b) is similar).
If $X\sm Z=\emptyset$ then $X=Z$ and so~$X$ is independent. But $|X|=\lfloor n/3\rfloor+1$ which
contradicts Fact~\ref{fact:ind}. So assume that $X\setminus Z\neq\emptyset$
and let $x'\in X\setminus Z$ be such that $d_{X\sm Z}^+(x')$ is minimal.
Fact~\ref{fact:avg} implies that
\[
d_{\bar{X\sm Z}}^+(x')>\delta^0(G)-(|X|-|Z|)/2\geq(\lfloor n/3\rfloor+1+|Z|)/2.
\]
By assumption $x'$ has no outneighbours in~$Y$,
so $d_{\bar{X\sm Z}}^+(x')=d_{\bar{X\cup Y}}^+(x')$ and thus~(a) holds.
\end{proof}

In order to prove the cases $\ell=6$ and $\ell \ge 7$
of Theorem~\ref{thm:short_cycles} we need some
more notation. An \emph{$xy$-butterfly}
is an oriented graph with vertices $x,y,z,a,b$ such that $xa$,
$xz$, $az$, $zb$, $zy$, $by$ are all the edges (Figure~\ref{fig:butterfly}).
\begin{figure}
\centering
\includegraphics[width=4cm]{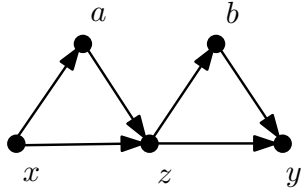}
\caption{An $xy$-butterfly}
\label{fig:butterfly}
\end{figure}
The crucial fact about a butterfly is that it contains
$x$-$y$ paths of lengths 2, 3 and 4, and is thus a useful tool in finding
cycles of prescribed length: any $y$-$x$ path of length $\ell-2$, $\ell-3$
or $\ell-4$ whose interior avoids the $xy$-butterfly yields an $\ell$-cycle containing~$x$.
The following fact tells us that a large minimum semidegree
guarantees the existence of a butterfly.

\begin{fact}\label{fact:butterfly}
If $G$ is an oriented graph on $n$ vertices
with $\delta^0(G)\geq \lfloor n/3\rfloor +1$ then
for any vertex $x\in V(G)$ there exists a vertex~$y$ such that $G$ contains
an $xy$-butterfly.
\end{fact}

\begin{proof}
By Fact~\ref{fact:ind} the outneighbourhood of~$x$ is not independent, so pick
an edge $az$ in it. Reapply Fact~\ref{fact:ind} to find an edge
$by$ in the outneighbourhood of $z$.
Note that as $x,a\in N^-(z)$ all the vertices are distinct.
\end{proof}

\begin{lemma}\label{lemma:6}
If~$G$ is an oriented graph on~$n\geq 6$ vertices with
$\delta^0(G)\geq \lfloor n/3\rfloor+1$ then for any vertex
$x\in V(G)$, $G$ contains a $6$-cycle through~$x$.
\end{lemma}

\begin{proof}
Fact~\ref{fact:butterfly} gives us an $xy$-butterfly for some vertex $y\in V(G)$,
with vertices $a,b,z$ as described in the definition of an $xy$-butterfly.
To complete the proof we may assume that each of the following holds.
\begin{itemize}
\item[(i)] There is no $y$-$x$ path of length 2.
\item[(ii)] There is no $y$-$x$ path of length 3 avoiding $a$.
\item[(iii)] There is no $y$-$x$ path of length 4 avoiding $z$.
\end{itemize}
Indeed, it is easy to check that if one of these does not hold then this $y$-$x$ path together
with a suitable subpath of the $xy$-butterfly forms the required cycle.%
    \COMMENT{Note that in case (i) we do not
need to exclude explicitly $a$ or $b$ or $z$ as the path having length 2 guarantees it automatically.
Similarly, we don't have to exclude~$z$ in case~(ii). }

Pick $Y\subseteq N^+(y)\sm \{a,x\}$, $X\subseteq N^-(x)\sm \{y\}$ such that%
     \COMMENT{As $n\geq6$ we have $Y\neq\emptyset$.}
$|Y|=\lfloor n/3\rfloor-1$ and $|X|=\lfloor n/3\rfloor$.
Observe that $b,z\not\in Y$ and $a,z\not\in X$. Moreover $X\cap
Y=\emptyset$ by~(i). Let $Y':=N^+(Y)\sm Y$, $X':=N^-(X)\sm X$.
Then $X\cap Y'=\emptyset$ and $Y\cap X'=\emptyset$ by~(ii).
Fact~\ref{fact:avg} implies that $|Y'| \geq \lfloor
n/3\rfloor/2+2$ and $|X'|\geq \lfloor n/3\rfloor/2+3/2$. By (i)
and the definitions of~$X$ and~$Y$ we have $x,y\not\in
X,Y,X',Y'$. Altogether this shows that
\begin{equation*}
n+|X'\cap Y'|\geq |X|+|Y|+|X'|+|Y'|+2 \geq 3\lfloor n/3\rfloor+9/2\geq (n-2)+9/2.
\end{equation*}
Hence $|X'\cap Y'|\geq 3$, and so $(X'\cap Y')\sm\{z\}\neq \emptyset$. But this implies that
there is a $y$-$x$ path of length~4 avoiding~$z$.
\end{proof}

The next two lemmas deal with the case $\ell\geq 7$.

\begin{lemma}\label{lemma:34or5}
Let $C$ be some positive integer. If $G$ is an oriented graph
on $n\geq 8 \cdot 10^{9} C$ vertices with $\delta^0(G)\geq
n/3-C+1$ then for every pair $x\neq y$ of vertices there exists
an $x$-$y$ path of length $3$, $4$ or~$5$.
\end{lemma}

\begin{proof}
Let $\epsilon := 1/10^4$ and $C':=10C/\eps$. Let~$X$ be a set of~$\lceil{n/3}\rceil-2C$
outneighbours of~$x$ in~$G-y$ and let~$Y$ be a set of~$\ceil{n/3}-2C$ inneighbours of~$y$ in~$G-x$,
chosen so that%
     \COMMENT{loose the $+1$ from the minimum degree since we might have that
$N^+(x)\sm\{y\}=N^-(y)\sm \{x\}$ and
$|N^+(x)|=|N^-(y)|=\delta^0(G)$ as well as $y\in N^+(x)$ and
$x\in N^-(y)$.} $|X\sm Y|$, $|Y\sm X|\geq C$. Let $Z:=X\cap Y$.
If there is an~$X$-$Y$ edge then we have an $x$-$y$ path of
length 3. So suppose there is no such edge. In particular this
implies that~$Z$ is independent and there are no~$X$-$Z$
or~$Z$-$Y$ edges.

Let $X':=N^+(X\sm Z)\sm X$ and $Y':=N^-(Y\sm Z)\sm Y$. Note
that $X'\cap Y=\emptyset$ and $Y'\cap X=\emptyset$, as
otherwise we have an $X$-$Y$ edge. Moreover, we may assume that
$X'\cap Y'=\emptyset$, as otherwise we have an $x$-$y$ path of
length~4. As no vertex in $X\sm Z$ has an outneighbour in~$Z$
we have $X'=N^+(X\sm Z)\sm (X\sm Z)$. Hence by
Fact~\ref{fact:avg}
\[
|X'| \geq \delta^0(G) - |X\sm Z|/2\geq \ceil{n/3}/2+|Z|/2.
\]
Similarly, $|Y'|\geq \ceil{n/3}/2+|Z|/2$. Observe that this implies
\begin{equation}\label{eq:XYX'Y'}
|V(G)\sm (X\cup X'\cup Y\cup Y')|\leq 4C.
\end{equation}
Note that
\begin{equation} \label{lowerX'}
|X'| \le n- |X \cup Y| - |Y'| \le n- (2n/3-|Z|-4C)-(n/6+|Z|/2) = n/6 +|Z|/2+4C.
\end{equation}
We call a vertex $x'\in X\sm Z$ \emph{good}
if $|N^+(x')\sm X|\geq n/6+|Z|/2-C'\geq |X'|-4C'/3$ (the last
inequality follows from~(\ref{lowerX'})).%
     \COMMENT{Also need $4C\le C'/3$.}
Suppose that at least
$\epsilon |X\sm Z|$ vertices in $X\sm Z$ are not good. Since $d^+_{X\sm Z}(x')\ge
\delta^0(G)-|N^+(x')\sm (X\sm Z)|=\delta^0(G)-|N^+(x')\sm X|$ for every $x'\in X\sm Z$
this implies that
\begin{align*}
e(X\sm Z)&\ge\epsilon |X\sm Z|(\delta^0(G)-(n/6+|Z|/2-C'))+(1-\epsilon)|X\sm Z|(\delta^0(G)-|X'|)\\
& \stackrel{(\ref{lowerX'})}{\geq}\epsilon |X\sm Z|(n/6-|Z|/2+C'/2)+ (1-\epsilon)|X\sm Z|(n/6-|Z|/2-5C)\\
& = |X\sm Z|(n/6-|Z|/2+\eps C'/2 -5C (1-\epsilon)) \\
&\ge |X\sm Z|(n/6-|Z|/2) \ge |X\sm Z|^2/2.
\end{align*}
But this is a contradiction as $G$ is an oriented graph. Thus we may assume that all but at most
$\epsilon |X\sm Z|$ vertices in $X\sm Z$ are good, and hence, since $|X'|\ge n/6\ge 4C'/(3\eps)$
we have%
     \COMMENT{We need $n\geq 8C'/\epsilon=8\cdot 10\cdot 10^8 C$ here.}
\begin{align}\label{eq:XXZbound}
e(X\sm Z,X') \geq (1-\epsilon)|X\sm Z|(|X'|-4C'/3)
&\geq (1-2\epsilon)|X\sm Z||X'|.
\end{align}
Call a vertex $x'\in X'$ \emph{nice} if $|N^-(x')\cap (X\sm Z)|\geq(1-2\sqrt{\epsilon})|X\sm Z|$. Then at least
$(1-2\sqrt{\epsilon})|X'|$ vertices in $X'$ are nice, as otherwise
\begin{align*}
e(X\sm Z,X')\leq 2\sqrt{\epsilon}|X'|(1-2\sqrt{\epsilon})|X\sm Z|+(1-2\sqrt{\epsilon})|X'||X\sm Z|
<(1-2\epsilon)|X'||X\sm Z|,
\end{align*}
which contradicts \eqref{eq:XXZbound}.
Consider a nice vertex%
     \COMMENT{exists since $(1-2\sqrt{\eps})|X'|\ge (1-1/50)n/6\ge 2$}
$x'\in X'\setminus\{y\}$. Note that $N^+(x')\cap (Y\cup Y')$ is either empty or equal
to~$\{x\}$ (as otherwise we get an $x$-$y$ path of length~4 or~5).
Since~$x'$ is nice it has at most $2\sqrt{\eps}|X\sm Z|$ outneighbours in~$X\sm Z$ and so%
     \COMMENT{Get $n/3-C+1-2\sqrt{\epsilon} n/3-1-4C\ge n/3-\sqrt{\epsilon}n$.
Ok if $5C<\sqrt{\epsilon}n/3$.}
\begin{equation}\label{eq:N+x'capX'}
|N^+(x')\cap X'|\stackrel{\eqref{eq:XYX'Y'}}{\geq}
	\delta^0(G)-2\sqrt{\epsilon} |X\sm Z| - 1-4C\geq n/3-\sqrt{\epsilon} n.
\end{equation}
In particular, $|X'|\geq n/3-\sqrt{\epsilon} n$. Similarly, $|Y'|\geq n/3-\sqrt{\epsilon} n$.
But $|X\cup Y|\geq n/3-C$ and so%
    \COMMENT{Requires $C\leq \sqrt{\epsilon} n$.}
\begin{equation}\label{upperX'}
|X'|\leq n-|X\cup Y|-|Y'|\leq n/3+2\sqrt{\epsilon}n.
\end{equation}
Now we combine this with the fact that at least%
     \COMMENT{ok if $\sqrt{\eps}|X'|\ge 1$, ie $\frac{1}{100}\frac{n}{6}\ge 1$.}
$|X'|-1-2\sqrt{\eps}|X|'\ge (1-3\sqrt{\eps})|X'|$ vertices in $X'\setminus \{y\}$ are nice to obtain
\[ |X'|^2/2\ge e(X') \stackrel{\eqref{eq:N+x'capX'}}{\ge} (1-3\sqrt{\epsilon} )|X'|(n/3-\sqrt{\epsilon} n)
\stackrel{(\ref{upperX'})}{\ge} (1-3\sqrt{\epsilon})|X'|(|X'|-3\sqrt{\epsilon}  n)>2|X'|^2/3.\]
This contradiction completes the proof.%
\COMMENT{Final contradiction as $(1-3\sqrt{\epsilon} )(1-3\sqrt{\epsilon} n/|X'|)\ge
(1-3\sqrt{\epsilon} )(1-9\sqrt{\epsilon}/(1-3\sqrt{\epsilon}))
= (1-3\sqrt{\epsilon} )(1-9\cdot (100/97)\sqrt{\epsilon} ) \ge 1-(3+10)\sqrt{\epsilon} =
87/100$.}
\end{proof}

\begin{lemma}\label{lemma:7}
Suppose~$\ell\geq 7$ and $n\geq 10^{10} \ell$. If~$G$ is an
oriented graph on~$n$ vertices with $\delta^0(G)\geq \lfloor n/3\rfloor +1$ then
for every vertex~$x\in V(G)$, $G$ contains an $\ell$-cycle through~$x$.
\end{lemma}

\begin{proof}
Fact~\ref{fact:butterfly} gives us an $xy$-butterfly
for some vertex $y\in V(G)$, with $a$, $b$ and $z$ as in the definition of an $xy$-butterfly.
Greedily pick a path~$P$ of length $\ell-7$ from~$y$ to some vertex~$v$ such that~$P$ avoids $a,b,x,z$
(the minimum semidegree condition implies the existence of such a path).

Now apply Lemma~\ref{lemma:34or5} to $G-(\{a,b,z\}\cup (V(P)\setminus \{v\})$ with
$C:=\ell$ (say) to find a $v$-$x$ path of length 3, 4 or 5.
Pick a path from~$x$ to~$y$ in the $xy$-butterfly of appropriate length
to obtain the desired $\ell$-cycle through~$x$.
\end{proof}

\section{Proofs of Theorem~\ref{thm:weakconj} and Proposition~\ref{digraphmindeg}} \label{partial}

The following lemma implies that if we allow ourselves a linear
`error term' in the degree conditions then instead of finding
an $\ell$-cycle, it suffices to look for a closed walk of
length~$\ell$. We will use (i) and (ii) in the proof of
Theorem~\ref{thm:weakconj}, (iii) in the proof of
Proposition~\ref{digraphmindeg} and (iv) in the proof of
Proposition~\ref{prop:arb_short}.
\begin{lemma}\label{lemma:walk2cycle}
Let~$\ell\ge 2$ be an integer.
\begin{enumerate}
\item[{\rm (i)}] Suppose that~$c>0$ and there exists an
    integer~$n_0$ such that every oriented graph~$H$
    on~$n\geq n_0$ vertices with $\delta^0(H)\geq cn$
    contains a closed walk~$W$ of length~$\ell$. Then for
    each~$\epsilon>0$ there
    exists~$n_1=n_1(\epsilon,\ell,n_0)$ such that if~$G$ is
    an oriented graph on $n\ge n_1$ vertices with
    $\delta^0(G)\geq (c+\epsilon)n$ then~$G$ contains an
    $\ell$-cycle.
\item[{\rm (ii)}] The analogue holds if we replace~$\delta^0(H)$ by $\delta^+(H)$
and $\delta^0(G)$ by~$\delta^+(G)$.
\item[{\rm (iii)}] The analogue of~(i) holds if we consider
    directed graphs instead of oriented graphs.
\item[{\rm (iv)}] The analogue of~(i) holds if we ask for a
copy of some specific (not necessarily closed) walk~$W$ of length~$\ell$ and for an
    orientation of a cycle which has a homomorphism
    into~$W$.
\end{enumerate}
\end{lemma}
Note that~(iv) is actually a strengthening of~(i).
The proof of Lemma~\ref{lemma:walk2cycle} is a standard
application of the Regularity lemma for digraphs. So we omit
the details, which can be found in~\cite{Lukethesis}. As
mentioned in the introduction, it would be interesting to find
a proof which avoids the Regularity lemma. This would probably
yield a much better bound on~$n_1$.

\removelastskip\penalty55\medskip\noindent{\bf Sketch of proof
of Lemma~\ref{lemma:walk2cycle}. } We only consider~(i). (The
arguments for the remaining parts are similar.) A directed
version of the Regularity lemma was proved by Alon and
Shapira~\cite[Lemma 3.1]{alonshapira}. Apply the degree form of
this Regularity lemma to~$G$ to obtain a partition of~$V(G)$
into clusters and a reduced digraph~$R'$. ($R'$ is sometimes
also called the cluster digraph). Roughly speaking, the
vertices of $R'$ are the clusters and there is a directed edge
from $A$ to $B$ in $R'$ if the bipartite subdigraph of $G$
consisting of the edges from $A$ to $B$ is $\eps'$-regular and
has density at least~$d$, where $\eps' \ll d \ll \eps$. One can
show that~$R'$ almost inherits the minimum semidegree of~$G$,
i.e.~$\delta^0(R')\ge (c+\eps/2)|R'|$. However, $R'$ need not
be oriented. But for every double edge of~$R'$ one can delete
one of the two edges randomly (with suitable probability) in
order to obtain an oriented spanning subgraph~$R$ of~$R'$ which
still satisfies $\delta^0(R)\ge c|R|$
(see~\cite[Lemma~3.1]{kelly_kuhn_osthus_hc_orient} for a
proof). Applying our assumption with $H:=R$ gives a closed walk
of length~$\ell$ in~$R$. Since~$n_1$ is large compared
to~$\ell$, this also holds for size of the clusters. So we can
apply the Embedding lemma (also called Key lemma) to find an
$\ell$-cycle
in~$G$. For a statement and proof of the Embedding lemma, see e.g.~the survey~\cite{regsurvey}.%
    \COMMENT{For (iii) we work with~$R'$. By hypothesis~$R'$ contains a closed walk $W$ of length~$\ell$.
Split each cluster of~$W$ into $\ell$ parts of (almost) equal
size. Then $W$ corresponds to an $\ell$-cycle in the digraph
whose vertices are the $|W|\ell$ subclusters and in which there
is an edge $V'_iV'_j$ iff $V_iV_j$ is an edge of $R'$ (here
$V'_i$ is a subcluster of $V_i$). The Embedding lemma applied
to this $\ell$-cycle gives an $\ell$-cycle in~$G$.}
\endproof

\removelastskip\penalty55\medskip\noindent{\bf Proof of
Theorem~\ref{thm:weakconj}(i). } Note that
Lemma~\ref{lemma:walk2cycle}(ii) implies that in order to prove
part~(i) it suffices to show that every oriented graph~$H$ with
$\delta^+(H)\ge |H|/k+149|H|/k^2$ contains a closed walk of
length~$\ell$. Theorem~\ref{thm:shen} implies that~$H$ contains
an
$a$-cycle~$C$ for some%
    \COMMENT{get $1/(1/k+149/k^2)+74=k^2/(k+149)+74 < k$. The latter holds
if $k^2+74k+149\cdot 74< k^2+149k$, ie $75k>149\cdot 74$, ok if
$k\ge 150$.} $a\le 1/(1/k+149/k^2)+74 < k$. But $a> 2$
since~$H$ is oriented and thus~$a$ divides~$\ell$ by our
definition of~$k$. By traversing~$C$ precisely $\ell/a$ times
we obtain the required closed walk of length~$\ell$ in~$H$.
\endproof

Note that the proof actually shows the following: Let~$c$ be
such that every oriented graph~$G$ with $\delta^+(G) \ge d$ has
a cycle of length at most $\lceil cn/d \rceil$. Then for each
$\eps >0$ there exists $n_0=n_0(\epsilon,\ell)$ such that every
oriented graph~$G$ on $n\geq n_0$ vertices with
$\delta^+(G)\geq cn/(k-1)+\epsilon n$ contains an $\ell$-cycle
(where~$\ell$ and~$k$ are as in Theorem~\ref{thm:weakconj}). In
particular, if we assume the Caccetta-H\"aggkvist conjecture,
then this implies that Conjecture~\ref{con:short} is
approximately true if we replace~$k$ by~$k-1$. Similarly, the
result in~\cite{chvatal_szemeredi} which gives a cycle of
length at most $\lceil 2n/(d+1)\rceil$ in an oriented graph of
minimum outdegree at least~$d$
implies that we may take $c:=2$.%
           \COMMENT{actually, a bit more}
It would be interesting to find improved approximate versions of Conjecture~\ref{con:short}.

To prove parts~(ii) and~(iii) of Theorem~\ref{thm:weakconj}, we will use the following lemma.

\begin{lemma}\label{lemma:diam6}
Let~$G$ be an oriented graph on~$n$ vertices.
\begin{enumerate}
\item[{\rm (i)}] If $\delta^0(G)\geq n/4$ then
either the diameter of~$G$ is at most~$6$ or~$G$ contains a $3$-cycle.
\item[{\rm (ii)}] If $\delta^0(G)> n/5$ then
either the diameter of~$G$ is at most~$50$ or~$G$ contains a $3$-cycle.
\end{enumerate}
\end{lemma}
\begin{proof}
We first prove~(i).
Consider $x\in V(G)$ and define $X_1:=N^+(x)$ and $X_{i+1}:=N^+(X_i)\cup X_i$ for
$i\ge 1$. If there exists an~$i$ with $\delta^+(G[X_i])>3|X_i|/8$ then~$G[X_i]$
contains a $3$-cycle by Theorem~\ref{shen_triangle}.
So assume not. Then there exists a vertex $x_i\in X_i$ with
$|N^+(x_i)\cap X_i|\leq 3|X_i|/8$. Hence
\begin{equation*}
|X_{i+1}|\ge |X_i|+(\delta^0(G)-3|X_i|/8)\ge 5|X_i|/8+n/4.
\end{equation*}
In particular $|X_2|\geq 13n/32$ and $|X_3|\ge 65n/256 + n/4 = 129n/256 > n/2$.
Similarly, for any vertex~$y\neq x$ we have that $|\{v\in V(G):dist(v,y)\leq 3\}|>n/2$,
and thus there exists an $x$-$y$ path of length at most~6, which completes the proof of~(i).

To prove~(ii), define sets~$X_i$ as before. Consider any~$i$ for which $|X_i|\le n/2$.
Similarly as before
\begin{align*}
|X_{i+1}|& \ge |X_i|+(\delta^0(G)-3|X_i|/8)> |X_i|+(n/5-3|X_i|/8) \ge |X_i|+(n/5-3n/16)\\
& =|X_i|+n/80.
\end{align*}
Thus $|X_{25}|>n/2$. Similarly, for any vertex~$y\neq x$ we have that
$|\{v\in V(G):dist(v,y)\leq 25\}|>n/2$. Thus there exists an $x$-$y$ path of length at most~50.
\end{proof}

\removelastskip\penalty55\medskip\noindent{\bf Proof of
Theorem~\ref{thm:weakconj}(ii). } As in the proof of~(i), by
Lemma~\ref{lemma:walk2cycle}(i) it suffices to show that every
sufficiently large oriented graph~$H$ with $\delta^0(H)\ge
|H|/4+1$ contains a closed walk of length~$\ell$. If~$H$ has a
$3$-cycle then it contains such a walk since~$3$ divides~$\ell$
by definition of~$k$. Thus we may assume that~$H$ has
no~$3$-cycle. Fact~\ref{fact:ind} implies that the maximum size
of an independent set is smaller than the
neighbourhood~$N_H(v)$ of any vertex~$v$. Thus~$H$ contains
some orientation of a triangle. By assumption this is not a
$3$-cycle, and so it must be transitive, i.e.~the triangle
consists of vertices $x,y,z$ and edges $xz,xy,zy$.

Since $H-z$ has no $3$-cycle, Lemma~\ref{lemma:diam6}(i)
implies that~$H-z$ contains a $y$-$x$ path~$P$ of length~$t\le
6$. This gives us 2 cycles $C_1:=yPxy$ and $C_2:=yPxzy$ of
lengths~$t+1$ and~$t+2$ respectively. Write~$\ell$ as
$\ell=a(t+1)+r$ with $0\leq r\leq t\leq 6$. We can wind~$r$
times around~$C_2$ and~$(a-r)$ times around~$C_1$ to find a
closed walk of length~$\ell$ in~$H$ provided that $r\leq a$.
But the latter holds as $a = \floor{\ell/(t+1)} \geq 6$.
\endproof

In the proof of Theorem~\ref{thm:weakconj}(iii), we will use the following result (on undirected graphs)
of Andr\'asfai, Erd\H{o}s and S\'os~\cite{andrasfai}:
\begin{theorem} \label{andrasfai}
Every triangle-free graph~$F$ on $n$ vertices with minimum degree $\delta(F) > 2n/5$ is
bipartite.
\end{theorem}

\removelastskip\penalty55\medskip\noindent{\bf Proof of Theorem~\ref{thm:weakconj}(iii). }
Again, by Lemma~\ref{lemma:walk2cycle}(i) it suffices
to show that every sufficiently large oriented graph~$H$ on $n$ vertices with $\delta^0(H)> n/5+1$
contains a closed walk of length~$\ell$.

Let $F$ be the underlying undirected graph of $H$. Since $H$
has no double edges, we have $\delta(F)> 2n/5$. Suppose first that~$F$ contains a triangle.
This cannot correspond to a $3$-cycle in~$H$, as this in turn immediately yields
a closed walk of length $\ell$ in~$H$. So~$H$ must contain a transitive triangle,
i.e.~vertices $x,y,z$ with $xz,xy,zy\in E(H)$.
We can now proceed similarly as in the proof of Theorem~\ref{thm:weakconj}(ii):
by Lemma~\ref{lemma:diam6}(ii) we can find a $y$-$x$ path~$P$
of length~$t\leq 50$ in $H-z$. This gives us 2 cycles
$C_1:=yPxy$ and $C_2:=yPxzy$ of lengths~$t+1$ and~$t+2$ respectively.
To obtain a closed walk of length $\ell$, write~$\ell$ as
$\ell=a(t+1)+r$ with $0\leq r\leq t\leq 50$. We can wind~$r$ times around~$C_2$
and~$(a-r)$ times around~$C_1$ to find a closed walk of length~$\ell$ in~$H$
provided that $r\leq a$.
But the latter holds as $a = \floor{\ell/(t+1)} \geq 50$.

So now suppose that $F$ does not contain a triangle. Then
Theorem~\ref{andrasfai} implies that $F$ (and thus $H$) is
bipartite. We will now use this to find a $4$-cycle in~$H$.
(This immediately yields a closed walk of length~$\ell$
in~$H$.) So suppose that~$H$ has no $4$-cycle. Write
$\delta_0:=\lceil n/5 \rceil +1$. Denote the vertex classes of
$H$ by $A$ and~$B$. Let $a:=|A|$ and $b:=|B|$, where without
loss of generality we have $b \le n/2$. On the other hand $b\ge
\delta(F) \ge 2n/5$ and so  $a \le 3n/5$. Now consider any $v
\in A$. Choose a set $X_1 \subseteq N^+(v)$ and $Y_1 \subseteq
N^-(v)$ with $|X_1|=|Y_1|=\delta_0$. Let $X_2:= N^+(X_1)$ and
$Y_2:= N^-(Y_1)$. Note that $X_2$ and $Y_2$ are disjoint, as
otherwise we would have a $4$-cycle (through~$v$) in~$H$. The
number of edges from $X_1$ to $X_2$ is at least
$|X_1|\delta_0$, so by averaging there is a vertex $x\in X_2$
which receives at least $|X_1|\delta_0/|X_2|$ edges from $X_1$.
This in turn means that $x$ sends at most
$|X_1|(1-\delta_0/|X_2|)$ edges to $X_1$. Recall that $x$ does
not send an edge to $Y_1$ since otherwise $x\in X_2\cap
Y_2=\emptyset$. So if we let $Z:=B \setminus (X_1 \cup Y_1)$,
then $x$ sends at least
$\delta_0-|X_1|(1-\delta_0/|X_2|)=\delta_0^2/|X_2|$ edges to
$Z$. In particular, $|Z| \ge \delta_0^2/|X_2|$. On the other
hand, $|Z|=b-2\delta_0\le n/10$. So $|X_2| \ge
\delta_0^2/(n/10) \ge 2\delta_0$. Since $X_2$ and $Y_2$ are
disjoint, this implies that $|Y_2| \le a -|X_2| \le
3n/5-2\delta_0<n/5$. On the other hand, the definition of $Y_2$
implies that $|Y_2|\ge \delta^0(H)$, a contradiction.
\endproof


\removelastskip\penalty55\medskip\noindent{\bf Proof of
Proposition~\ref{digraphmindeg}. } First suppose that~$\ell$ is
even. The inequality $\delta_{di}(\ell,n) \ge
\delta_{orient}(\ell,n)$ is trivial. For the upper bound on
$\delta_{di}(\ell,n)$, suppose we are given a digraph~$H$
on~$n$ vertices with $\delta^0(H) \ge \delta_{orient}(\ell,n)$.
If $H$ has a double edge, it has a closed walk of length
$\ell$. If it has no double edge, then $H$ has an $\ell$-cycle
by definition of $\delta_{orient}(\ell,n)$. So in both cases,
$H$ has a closed walk of length~$\ell$. So part (iii) of
Lemma~\ref{lemma:walk2cycle} implies that for each $\eps>0$
there is an $n_0$ so that for all $n \ge n_0$ we have
$\delta_{di}(\ell,n) \le \delta_{orient}(\ell,n)+\eps n$, as
required.

If $\ell$ is odd, we obtain the lower bound by considering the
complete bipartite digraph with vertex class sizes as equal as
possible. The upper bound follows e.g.~from~(\ref{density}).
\endproof

\section{Proofs of results on arbitrary orientations}\label{sec:arb}

\subsection{Proof of Proposition~\ref{prop:arb_short}}\label{subsec:arb_short}

For both parts of Proposition~\ref{prop:arb_short}, the proof
divides into three steps.
\begin{enumerate}
\item For a given $\ell$-cycle~$C$ with cycle-type~$k$ find
    an appropriate walk~$W$ with prescribed orientation
    (which will be a cycle for~$k\geq 3$) into which there
    is a digraph homomorphism of~$C$.
\item Prove that the minimum semidegree condition in Proposition~\ref{prop:arb_short}
guarantees a copy of~$W$ in any sufficiently large oriented graph~$G$.
\item Apply Lemma~\ref{lemma:walk2cycle}(iv) to `lift' this
    result to one on the cycle~$C$ itself.
\end{enumerate}

Let us start with the first step.
For~$k=0$ it is clear that there is a digraph homomorphism
of~$C$ into a directed path of length~$\ell$. For~$k\geq 3$ we can
let~$W$ be a directed~$k$-cycle.
Suppose that~$k=1$. Then the number of edges of~$C$ oriented forwards is one
larger than the number of its edges oriented backwards. So~$C$ must contain a subpath
of the form~\texttt{ffb}, where we
write~\texttt{f} for an edge oriented forwards and~\texttt{b}
for an edge oriented backwards. But this means that there
exist constants~$0\leq k_1,k_2<\ell$ (depending on~$C$) such that there is a
digraph homomorphism of~$C$ into the oriented walk~$W$ obtained
by adding a transitive triangle to the~$k_1$th vertex of a
directed path of length~$k_2$ (see Figure~\ref{subfig:type_1}).
\begin{figure}
\centering\footnotesize
\subfigure[$k=1$, \texttt{ffb} in cycle]
{  \includegraphics[width=4cm]{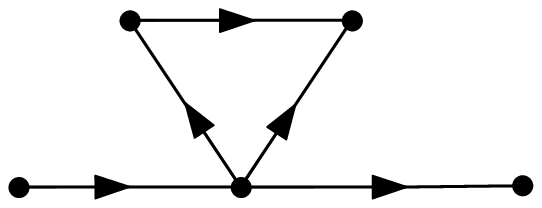}\label{subfig:type_1}}
\hspace{0.5cm}
\subfigure[$k=2$, \texttt{fffb} in cycle]
{  \includegraphics[width=3.5cm]{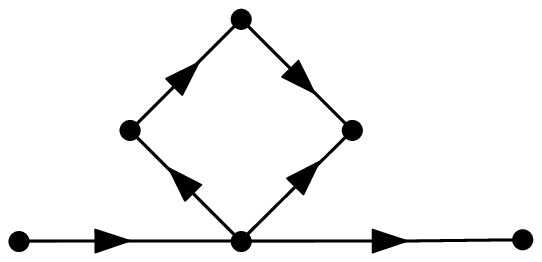}\label{subfig:type_2_square}}
\hspace{0.5cm}
\subfigure[$k=2$, \texttt{ffb} in cycle twice]
{  \includegraphics[width=4.5cm]{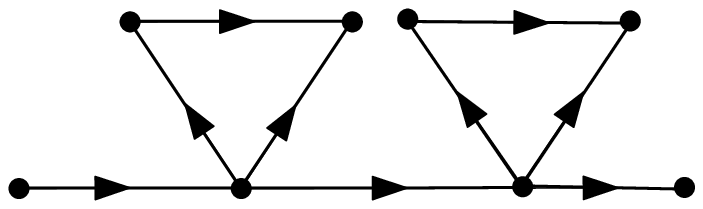}\label{subfig:type_2_triangle}}
\caption{The walks needed in the cases~$k=1$ and~$k=2$.}
\label{fig:W}
\end{figure}

Finally suppose that~$k=2$. So~$C$
contains two more edges oriented forwards than backwards.
Hence~$C$ contains two subpaths of the form~\texttt{ffb} or one subpath
of the form~\texttt{fffb}. In the first case we take~$W$ to be a suitable directed path
of length less than~$\ell$ with two transitive triangles attached, possibly to the same vertex (see
Figure~\ref{subfig:type_2_triangle}). In the second case we
let~$W$ be a suitable directed path with a 4-cycle
oriented~\texttt{fffb} attached (see Figure~\ref{subfig:type_2_square}).

For the second step we have to show that the relevant minimum
semidegree condition implies the existence of~$W$ in~$G$.
If~$W$ is a path then we only need the minimum semidegree
to be at least~$\ell$. If~$W$ is a~$k$-cycle then we just apply
Theorem~\ref{thm:short_cycles}. So suppose that $k=1,2$ and
consider any vertex $x$ of~$G$. The minimum semidegree
condition~$\delta^0(G)\geq (1/3+\alpha)n$ implies each vertex
$y\in N^+(x)$ has at least~$\alpha n$ neighbours in~$N^+(x)$.
So~$G$ contains the transitive triangles needed in
Figures~\ref{fig:W}(a) and~(c). To see that we can also find
the 4-cycle oriented \texttt{fffb}, suppose that~$N:=G[N^+(x)]$
does not contain a directed path of length~2 (otherwise we are
done). Then $N$ must contain two distinct vertices $y$ and $y'$ such
that $y$ has no outneighbours in~$N$ and $y'$ has no
inneighbours in~$N$. But this means that there is some $z\in
N^+(y)\cap N^-(y')$ and then $xyzy'$ has the required
orientation \texttt{fffb}. Hence we can find any of the walks
in Figure~\ref{fig:W} greedily. An application of
Lemma~\ref{lemma:walk2cycle}(iv) now completes the proof of
Proposition~\ref{prop:arb_short}. The argument for the
case~$k\geq 3$ also shows that Conjecture~\ref{con:short} would
imply an approximate version of Conjecture~\ref{con:arb_short}.

\subsection{Proof of universal pancyclicity result}\label{subsec:arb_pan}

To deduce Theorem~\ref{thm:arbpan} from
Theorem~\ref{thm:arb_ham} and Proposition~\ref{prop:arb_short} we will use the following
observation which is similar to one in~\cite{KO_3uniform}.

\begin{lemma}\label{lemma:random_split}
There exists an integer~$n_1$ such that the following holds for
all~$0<\alpha< 1$. Suppose we are given an oriented graph~$G$
on~$n\geq n_1$ vertices with minimum
semidegree~$\delta^0(G)\geq(3/8+\alpha-n^{-3/8})n$
where~$n/2\in\N$. Then there is a subset $U\subseteq V(G)$
of size $|U|=n/2:=u$ such that $\delta^0(G[U])\geq(3/8+\alpha-u^{-3/8})u$.
\end{lemma}

To prove it we need a large deviation bound for the
hypergeometric distribution (see e.g.~\cite[Theorem
2.10]{JLR_random_graphs}).

\begin{lemma}\label{lemma:hyper}
Given~$q\in \N$ and sets~$A\subseteq T$ with~$|T|\geq q$,
let~$Q$ be a subset of size $q$ of~$T$ chosen uniformly at random.
Let~$X:=|A\cap Q|$. Then for all~$0<\eps<1$ we
have
\[
\Prob[|X-\E (X)|\geq \eps\E(X)]\leq 2\exp\left(-\frac{\eps^2}{3}\E(X)\right).
\]
\end{lemma}

\medskip\cproof{Proof of Lemma~\ref{lemma:random_split}. }{ Consider a
subset $U$ of vertices of~$G$ chosen uniformly at random from all subsets of~$V(G)$
of size~$u$. Let~$\eps:=(1-2^{-3/8})u^{-3/8}$. Consider any
vertex~$x$ of~$G$ and define a random
variable~$X^+:=|N^+(x)\cap U|$. Observe that~$\eps\E(X^+)\leq
\eps u =(1-2^{-3/8})u^{5/8}$ and hence
\begin{align*}
\E(X^+) \geq (3/8+\alpha-n^{-3/8})u = (3/8+\alpha-u^{-3/8})u +\eps\E(X^+).
\end{align*}
Then by Lemma~\ref{lemma:hyper} we have
\begin{align*}
\Prob[X^+\leq (3/8+\alpha-u^{-3/8})u] \leq \Prob[X^+\leq (1-\eps)\E(X^+)]
\leq 2\exp\left(-\frac{(1-2^{-3/8})^2}{3u^{3/4}}\frac{u}{4}\right) \leq n^{-2}.
\end{align*}
The final inequality holds since we assume~$n$, and hence~$u$,
to be sufficiently large. The same bound holds when we consider
inneighbourhoods of vertices. Hence with positive probability there
exists a set $U\subseteq V(G)$ with the desired minimum
semidegree property. }

We are now in a position to derive Theorem~\ref{thm:arbpan}.

\medskip \cproof{Proof of Theorem~\ref{thm:arbpan}. }{ Given $\alpha>0$,
set~$\ell_0:=\max\{n_0(\alpha/3),n_1,(6/\alpha)^{8/3}\}$, where~$n_0$
is the function defined in Theorem~\ref{thm:arb_ham} and~$n_1$
is as in Lemma~\ref{lemma:random_split}. Let $n\gg \ell_0,1/\alpha$
and consider an oriented graph~$G$ on~$n$ vertices with
minimum semidegree~$\delta^0(G)\geq(3/8+\alpha)n$.
Choose any $3\le \ell\le n$ and any orientation~$C$ of an $\ell$-cycle.
We have to show that $G$ contains a copy of~$C$. This is clear
if $\ell\leq \ell_0$, since $n\gg \ell_0,1/\alpha$ and thus an
application of Proposition~\ref{prop:arb_short} gives us~$C$
immediately.

So we may assume that $\ell>\ell_0$. Let $k$ be an integer such that
$2^k\ell\le n<2^{k+1}\ell$. A straightforward application of
Lemma~\ref{lemma:hyper} implies the existence of a subgraph $G'$ of $G$ on $n':=2^k\ell$ vertices
with $\delta^0(G')\geq (3/8+\alpha/2)n'$. Apply Lemma~\ref{lemma:random_split}
$k$ times to obtain a subgraph $G''$ of $G'$ on $\ell$ vertices with
$\delta^0(G'')\geq (3/8+\alpha/2-\ell^{-3/8})\ell\ge (3/8+\alpha/3)\ell$.
Since~$\ell> n_0(\alpha/3)$ we can now
apply Theorem~\ref{thm:arb_ham} to obtain a Hamilton cycle
oriented as~$C$ in~$G''$ and hence the desired orientation of an
$\ell$-cycle in~$G$.
}


\medskip

{\footnotesize \obeylines \parindent=0pt

Luke Kelly, Daniela K\"{u}hn \& Deryk Osthus
School of Mathematics
University of Birmingham
Edgbaston
Birmingham
B15 2TT
UK
}

{\footnotesize \parindent=0pt

\it{E-mail addresses}:
\tt{\{kellyl,kuehn,osthus\}@maths.bham.ac.uk}}

\begin{thebibliography}{10}

\bibitem{ag} N. Alon and G. Gutin, Properly colored Hamilton cycles in edge colored
complete graphs, \emph{Random Structures and Algorithms}~{\bf 11} (1997), 179--186.

\bibitem{alonshapira} N.~Alon and A.~Shapira,
Testing subgraphs in directed graphs,
\emph{Journal of Computer and System Sciences}~{\bf 69} (2004), 354--382. %
\COMMENT{should change this ref in other papers too}

\bibitem{andrasfai} B.~Andr\'asfai, P.~Erd\H{o}s and V.T.~S\'os,
On the connection between chromatic number, maximal clique and minimal degree of a graph,
\emph{Discrete Math.} {\bf  8} (1974), 205--218.

\bibitem{BGbook} J.~Bang-Jensen and G.~Gutin,
\emph{Digraphs: Theory, Algorithms and Applications}, Springer 2000.

\bibitem{BCW} M.~Behzad, G.~Chartrand and C.~Wall,
On minimal regular digraphs with given girth, \emph{Fund. Math.}~\textbf{69} (1970), 227--231.

\bibitem{brownharary} W.G.~Brown and F.~Harary,
Extremal digraphs, Combinatorial theory and its applications, I (Proc. Colloq., Balatonf\"ured, 1969), 135--198. North-Holland, Amsterdam, 1970.

\bibitem{brownsimonovits} W.G. Brown and M. Simonovits,
Extremal multigraph and digraph problems, Paul Erd\H{o}s and his mathematics, II (Budapest, 1999),
{Bolyai Soc. Math. Stud.} {\bf 11} (2002), 157--203.

\bibitem{CH} L.~Caccetta and R.~H\"aggkvist,
On minimal graphs with given girth,
\emph{Proceedings of the ninth Southeastern Conference on Combinatorics,
Graph Theory and Computing, Congress. Numerantium XXI, Utilitas Math.} (1978), 181--187.

\bibitem{chvatal_szemeredi} V.~Chv\'atal and E.~Szemer\'edi, Short cycles in directed graphs,
\emph{J. Combinatorial Theory B}~\textbf{35} (1983), 323--327.

\bibitem{Darbinyan} S.~Darbinyan, Pancyclicity of digraphs with large semidegrees,
\emph{Akad. Nauk Armyan. SSR Dokl.}~\textbf{80} (1985), 51--54.

\bibitem{gh} A. Ghouila-Houri, Une condition suffisante d'existence d'un circuit
hamiltonien, \emph{C.R.~Acad.~Sci.~Paris}~\textbf{25} (1960), 495--497.

\bibitem{haggtom} R.~H\"aggkvist and C.~Thomassen,
On pancyclic digraphs,
\emph{J.~Combinatorial Theory B} {\bf 20} (1976), 20--40.

\bibitem{HHK_CH} P.~Hamburger, P.~Haxell and A.~Kostochka, On directed triangles in digraphs,
\emph{Electronic J. Combin.} {\bf  14}  (2007),   N19, 9 pages.

\bibitem{JLR_random_graphs} S.~Janson, T.~Luczak,
    A.~Ruci\'nski, \emph{Random Graphs}, Wiley-Interscience
    2000.

\bibitem{KKO_exact} P.~Keevash, D.~K\"uhn and D.~Osthus, An
    exact minimum degree condition for {H}amilton cycles in
    oriented graphs, \emph{Journal of the London Mathematical
    Society}~\textbf{79} (2009), 144--166.

\bibitem{Lukethesis} L.~Kelly,
Cycles in directed graphs, PhD thesis, University of Birmingham, 2009.

\bibitem{K_arb} L.~Kelly,
   Arbitrary orientations of Hamilton cycles in oriented graphs, submitted.

\bibitem{kelly_kuhn_osthus_hc_orient} L.~Kelly, D.~K\"uhn and
    D.~Osthus, A {D}irac type result on {H}amilton cycles in
    oriented graphs, \emph{Combin.~Probab.~Comput.}~\textbf{17} (2008),
    689--709.

\bibitem{regsurvey} J.~Koml\'os and M.~Simonovits, Szemer\'edi's Regularity
Lemma and its applications in graph theory, \emph{Bolyai Society
Mathematical Studies 2, Combinatorics, Paul Erd\H{o}s is Eighty (Vol.~2)}
(D.~Mikl\'os, V.~T.~S\'os and T.~Sz\H{o}nyi eds.), Budapest (1996),
295--352.

\bibitem{KO_3uniform} D.~K\"uhn and D.~Osthus, Loose Hamilton
    cycles in 3-uniform hypergraphs of large minimum degree,
    \emph{J. Combinatorial Theory Series B}~\textbf{96} (2006), 767--821.

\bibitem{survey} D. K\"uhn and D. Osthus, Embedding large subgraphs into dense graphs, in \emph{Surveys in Combinatorics}
(S. Huczynska, J.D. Mitchell and C.M. Roney-Dougal eds.),
\emph{London Math.~Soc.~Lecture Notes}~\textbf{365}, 137--167, Cambridge University Press, 2009.

\bibitem{KOT_hc} D.~K\"uhn, D.~Osthus and A.~Treglown,
    {H}amiltonian degree sequences in digraphs, submitted.

\bibitem{Nathanson} M.~Nathanson, The Caccetta-H\"aggkvist conjecture
and additive number theory, \emph{ArXiv Mathematics e-prints} (2006).

\bibitem{Nishimura} T.~Nishimura,
Short cycles in digraphs,
\emph{Proceedings of the First Japanese Conference on Graph Theory and Applications}
{\bf 72} (1988), 295--298.

\bibitem{shen_triangle} J.~Shen, Directed triangles in digraphs,
\emph{J.~Combinatorial Theory B}~\textbf{74} (1998), 405-407.

\bibitem{shen_2} J.~Shen,
On the girth of digraphs,
\emph{Discrete Math.} {\bf 211} (2000), 167--181.

\bibitem{shen_ch} J.~Shen, On the {C}accetta-{H}\"aggkvist conjecture,
\emph{Graphs and Combinatorics}~\textbf{18} (2002), 645--654.

\bibitem{Song} Z.M.~Song, Pancyclic oriented graphs,
\emph{J.~Combinatorial Theory B}~\textbf{18} (1994), 461--468.

\bibitem{tom} C. Thomassen, An Ore-type condition implying a digraph to be pancyclic,
\emph{Discrete Math.}~{\bf 19} (1977), 85--92.


\end{thebibliography}
\end{document}